\documentclass{amsart}

\usepackage{amsmath}
\usepackage{amssymb}
\usepackage{amsthm}

\numberwithin{equation}{section}

\newtheorem{theorem}              {Theorem}[section]
\newtheorem{lemma}       [theorem]{Lemma}
\newtheorem{proposition} [theorem]{Proposition}
\newtheorem{corollary}   [theorem]{Corollary}
\newtheorem{definition}  [theorem]{Definition}
\newtheorem{remark}      [theorem]{Remark}

\newtheorem{example}     [theorem]{Example}

\newcommand{\cmdrhu}          {\rightharpoonup}

\newcommand{\cmdYD}           {Y\!D}
\newcommand{\cmdcalA}         {\mathcal{A}}

\newcommand{\cmdcalD}         {\mathcal{D}}
\newcommand{\cmdcalE}         {\mathcal{E}}

\newcommand{\cmdcalI}         {\mathcal{I}}

\newcommand{\cmdcalR}         {\mathcal{R}}
\newcommand{\cmdcalS}         {\mathcal{S}}
\newcommand{\cmdcalYD}        {\mathcal{YD}}

\newcommand{\cmdbbA}          {\mathbb{A}}
\newcommand{\cmdbbD}          {\mathbb{D}}

\newcommand{\cmdbbQ}          {\mathbb{Q}}

\newcommand{\cmdbbm}          {\mathbf{m}}

\newcommand{\cmdfrakB}        {\mathfrak{B}}
\newcommand{\cmdfraka}        {\mathfrak{a}}

\newcommand{\cmdcop}          {\mathrm{cop}}

\newcommand{\cmdch}           {\mathop{\mathrm{ch}}}

\newcommand{\cmddet}          {\mathop{\mathrm{det}}}
\newcommand{\cmddiag}         {\mathop{\mathrm{diag}}}
\newcommand{\cmdgr}           {\mathop{\mathrm{gr}}}
\newcommand{\cmdid}           {\mathop{\mathrm{id}}}
\newcommand{\cmdmod}          {\mathop{\mathrm{mod}}}

\newcommand{\cmdres}          {\mathop{\mathrm{res}}}
\newcommand{\cmdth}           {\mathop{\mathrm{th}}}

\newcommand{\cmdHom}          {\mathop{\mathrm{Hom}}}
\newcommand{\cmdKer}          {\mathop{\mathrm{Ker}}}

\newcommand{\cmddotrtimes}    {\mathop{\raisebox{0.2ex}{\makebox[0.86em][l]{${\scriptstyle>\mathrel{\mkern-4mu}\lessdot}$}}\raisebox{0.12ex}{$ \shortmid$}}}
\newcommand{\cmdbicross}      {\mathop{\Join}\nolimits}
\newcommand{\cmdarrow}        {\mathop{\longrightarrow}}
\newcommand{\cmdbarotimes}    {\mathop{\overline{\otimes}}}
\newcommand{\cmdbarbigotimes} {\mathop{\overline{\bigotimes}}}

\begin{document}

\title     [Quasitriangular pointed Hopf algebras]{On minimal quasitriangular pointed Hopf~ algebras}

\author    {Akira Masuoka}
\address   {Institute of Mathematics, University of Tsukuba, Ibaraki 305-8571, Japan}
\email     {akira@math.tsukuba.ac.jp}

\dedicatory{Dedicated to the memory of Professor Takayoshi Wakamatsu}
\thanks    {}

\date      {}

% A1
\begin{abstract}
  The quantized enveloping algebra $U_q$ is constructed as a quotient of the generalized quantum double $ U^{\leq 0}_q \cmdbicross_{\tau} U^{\geq 0}_q $ associated to a natural skew pairing $ \tau : U^{\leq 0}_q \otimes U^{\geq 0}_q \rightarrow k $.
  This double is generalized by
  \begin{equation*}
    \cmdcalD = ( \cmdfrakB ( V ) \cmddotrtimes F ) \cmdbicross_{\tau} ( \cmdfrakB ( W ) \cmddotrtimes G ),
  \end{equation*}
  where $F$, $G$ are abelian groups,
  $ V \in {}^F_F \cmdcalYD $, $ W \in {}^G_G \cmdcalYD $ are Yetter-Drinfeld modules and $ \cmdfrakB ( V ) $, $ \cmdfrakB ( W ) $ are their Nichols algebras.
  We prove some results on Hopf ideals of $ \cmdcalD $,
  including a characterization of what we call thin Hopf ideals.
  As an application we give an explicit description of those minimal quasitriangular pointed Hopf algebras in characteristic zero which are generated by skew primitives.
\end{abstract}

\maketitle

\setcounter{section}{0}

% 1-1
\section{Introduction}

Throughout the paper we work over a fixed field $k$.

% 1-1a
To define the (standard) quantized enveloping algebra, $ U_q $, Joseph \cite{J} adopted an effective way.
That is
\begin{itemize}
  \item first to construct a natural skew pairing
  \begin{equation} \label{eq1.1}
    \tau : U^{\leq 0}_q \otimes U^{\geq 0}_q \rightarrow k
  \end{equation}
  between the non-positive and the non-negative parts of $U_q$,
  \item then to use this $ \tau $ to construct the generalized quantum double
  \begin{equation} \label{eq1.2}
    U^{\leq 0}_q \cmdbicross_{\tau}  U^{\geq 0}_q ,
  \end{equation}
  which deforms only by its product,
  the Hopf algebra $  U^{\leq 0}_q \otimes U^{\geq 0}_q $ of tensor product, and
  \item finally to divide this last double by some central grouplikes.
\end{itemize}

% 1-2
This construction can apply to all known analogues and generalizations of $ U_q $ which include especially the Frobenius-Lusztig kernel $u_q$,
a finite-dimensional analogue of $U_q$ at roots of unity.

Recall from the works (including \cite{AS1}, \cite{AS2}) by Andruskiewitsch and Schneider that the notion of Nichols algebras gives a sophisticated viewpoint to study quantized enveloping algebras.
Let $G$ be an abelian group,
and let $W$ be an object in the braided tensor category $ {}^G_G \cmdcalYD $ of Yetter-Drinfeld modules over $G$.
The {\it Nichols algebra} $ \cmdfrakB ( W ) $ of $W$ (see \cite{AS1}) is a braided graded Hopf algebra in $ {}^G_G \cmdcalYD $ which has a certain functorial property.
% 1-3
The bosonization of $ \cmdfrakB ( W ) $ by $G$ forms an ordinary graded Hopf algebra,
which we denote by $ \cmdfrakB ( W ) \cmddotrtimes G $;
the graded Hopf algebra of this form is characterized as a coradically graded pointed Hopf algebras generated by skew primitives.
We assume that $W$ is {\it of diagonal type} in the sense that it is a (direct) sum of one-dimensional subobjects.
The Nichols algebra $ \cmdfrakB ( W ) $ generalizes the plus part $ U^+_q $ of $ U_q $,
while $ \cmdfrakB ( W ) \cmddotrtimes G $ generalizes $ U^{\geq 0}_q $.
(To be more precise, we need to replace the Nichols algebra with the looser,
{\it pre-Nichols algebra} \cite{M1}, to include those $ U_q $ at roots of unity.)
Let $F$ be another abelian group, and let $V$ be an object in $ {}^F_F \cmdcalYD $ of diagonal type.
% 1-4
Let $ \tau_0 : F \times G \rightarrow k ^{\times} ( = k \setminus 0 ) $ be a bimultipicative map,
and suppose that $V$ and $W$ are, roughly speaking, dual to each other, compatibly with $ \tau_0 $.
Then by \cite{M1}, $ \tau_0 $ extends to a skew pairing
\begin{equation*}
  \tau : ( \cmdfrakB ( V ) \cmddotrtimes F ) \otimes ( \cmdfrakB ( W ) \cmddotrtimes G ) \rightarrow k
\end{equation*}
which fulfills some reasonable requirements;
this generalizes the $ \tau $ in (\ref{eq1.1}).
This generalized $ \tau $ deforms the tensor product of the Hopf algebras $ \cmdfrakB ( V ) \cmddotrtimes F $, $ \cmdfrakB ( W ) \cmddotrtimes G $ into
\begin{equation*}
  \cmdcalD = ( \cmdfrakB ( V ) \cmddotrtimes F ) \cmdbicross_{\tau} ( \cmdfrakB ( W ) \cmddotrtimes G ),
\end{equation*}
which generalizes the double given in (\ref{eq1.2}).
We call $ \cmdcalD $ the {\it generalized quantum double} associated to $ \tau $.

% 1-5
Note that $ \cmdfrakB ( V ) \cmddotrtimes F $ generalizes $ U^{\leq 0}_q $,
but $ \cmdfrakB ( V ) $ does not correspond to the minus part $ U^-_q $,
but to its image by the antipode.
Thus, our construction of $ \cmdcalD $ from a pair of Nichols algebras is more symmetric,
compared with the construction of the double (\ref{eq1.2}) which treats with $ U^{\pm}_q $.
This symmetry makes arguments much simpler, I believe.

If $ \cmdfrakB ( W ) \cmddotrtimes G $ is finite-dimensional, and $ \tau_0 $ is non-degenerate,
then $ \cmdcalD $ concides with the quantum double of $ \cmdfrakB ( W ) \cmddotrtimes G $ as Drinfeld originally defined,
which we denote by $ D ( \cmdfrakB ( W ) \cmddotrtimes G ) $.
It is quite noteworthy that in characteristic zero,
Heckenberger \cite{H} classified all objects $W$ of diagonal type such that $ \cmdfrakB ( W ) $ is finite-dimenstional.
Also, when $k$ is an algebraically closed field of characteristic zero,
Andruskiewitsch and Schneider \cite{AS2} classified all finite-dimensional pointed Hopf algebras $ A $,
necessarily including those of the form $ \cmdfrakB ( W ) \cmddotrtimes G $,
such that all prime divisors of the order $ \left| G ( A ) \right| $ of grouplikes are greater than $7$.

% 1-6
The detailed construction of $ \cmdcalD $ will be given in Section 3.
The preceding Section 2 is devoted to preliminaries on Yetter-Drinfeld modules and (pre-)Nichols algebras.
Our aim of this paper is to study Hopf ideals of $ \cmdcalD $,
and accordingly its quotient Hopf algebras.
Note that $ \cmdcalD $ is a pointed Hopf algebra in which the group $ G ( \cmdcalD ) $ of all grouplikes equals $ F \times G $.
Let $C$ be the subgroup of $ F \times G $ consisting of all central grouplikes in $ \cmdcalD $,
and let $ \cmdcalD_c = \cmdcalD / ( c - 1 \mid c \in C ) $ denote the associated, quotient Hopf algebra.
In Section 4, which is independent of the following sections,
we assume that $F$, $G$ are finite groups with $\tau_0$ non-degenerate,
and prove under some additional mild restriction that $ \cmdcalD $ is isomorphic as an algebra to the tensor product $ kC \otimes \cmdcalD_c $ if and only if $ CP = F \times G $,
% 1-7
where $ P = C^{\perp} $ denotes the orthogonal subgroup (with $F \times G$ identified with the dual group $ ( F \times G )^{\wedge} $ via $ \tau_0 $); see Theorem \ref{4.2}.
These equivalent conditions are described in terms of the associated Cartan matrix,
when $\cmdcalD$ is of a special form such that $ \cmdcalD_c $ is closely related to the Frobenius-Lusztig kernel $ u_q $;
see Corollary \ref{4.3}.
(These two results are generalized formulations of results by Leonid Krop \cite{Kr},
and will be proved in the generalized situation with $ \cmdfrakB ( V ) $, $ \cmdfrakB ( W ) $ replaced by pre-Nichols algebras.)
We say that a Hopf ideal $ \cmdfraka \subset \cmdcalD $ is {\it thin} if $ V \cap \cmdfraka = 0 = W \cap \cmdfraka $.
In Section 5 we describe, under some mild assumption, these thin Hopf ideals $ \cmdfraka \subset \cmdcalD $ and the corresponding quotient Hopf algebras $ \cmdcalD / \cmdfraka $; see Theorem \ref{5.4}.
% 1-8
The result says that $ \cmdfraka $ is of the form $ \cmdfraka = \cmdfraka ( T , Z , \zeta ) $, parametrized by three data,
$T$, $Z$, $ \zeta $; among these $T$ is a subgroup of $C$, and $ \cmdfraka ( T , 0 , 0 ) = ( t - 1 \mid t \in T ) $ if $Z$, $ \zeta $ are zero.
This result will be applied in the following two sections to prove Theorems \ref{6.1} and \ref{7.2}.
Theorem \ref{6.1} tells us that if $ \cmdcalD $ satisfies a `conectedness' assumption,
every Hopf ideal $ \cmdfraka $ of $ \cmdcalD $ is either so small that it is of the form $ \cmdfraka ( T , 0 , 0 ) $ (see above),
or so large that $ \cmdcalD / \cmdfraka $ is a group algebra.
As an application we describe all Hopf ideals of the quantized enveloping algebra $ U_q $ associated to such a symmetrizable Borcherds-Cartan matrix that is irreducible; see Example \ref{6.2}.
% 1-9
Theorem \ref{7.2} shows, assuming that $k$ is an algebraically closed field of characteristic zero,
that if $A$ is such a minimal quasitriangular pointed Hopf algebra that is generated by skew primitives,
then it is of the form $ \cmdcalD / \cmdfraka ( T , Z , \zeta ) $, where $ \cmdcalD $ is Drinfeld's quantum double $ D ( \cmdfrakB ( W ) \cmddotrtimes G ) $ of a finite-dimensional $ \cmdfrakB ( W ) \cmddotrtimes G $.
This formulation turns nicer, in virtue of \cite{AS2},
if we assume that all prime divisors of $ \left| G ( A ) \right| $ are greater than $7$; see Corollary \ref{7.3}.
In the last Section 8, we reproduce from Theorem \ref{7.2} Gelaki's classification result \cite{G} of minimal triangular pointed Hopf algebras in characteristic zero.

% 2-1
\section{Preliminaries}

\subsection{}

Let $J$ be a Hopf algebra with bijective antipode. The coproduct, the counit and the antipode (for any Hopf algebra) are denoted by
\begin{equation} \label{eq2.1}
  \Delta ( x ) = x_1 \otimes x_2 , \quad \varepsilon \quad \mbox{and} \quad \cmdcalS ,
\end{equation}
respectively.
Let $V$ be a {\it Yetter-Drinfeld module} over $J$.
This means that $V$ is a left $J$-module and left $J$-comodule,
whose structures we denote by
\begin{equation} \label{eq2.2}
  x \cmdrhu v \quad ( x \in J , v \in V ), \quad \rho ( v ) = v_{-1} \otimes v_0,
\end{equation}
and satisfies
\begin{equation*}
  \rho ( x \cmdrhu v ) = x_1 v_{-1} \cmdcalS ( x_3 ) \otimes ( x_2 \cmdrhu v_0 ) \quad ( x \in J , v \in V ).
\end{equation*}
Let $ {}_J^J \cmdcalYD $ denote the category of all Yetter-Drinfeld modules over $J$.
%2-2
This naturally forms a braided tensor category.

Two objects $V$, $W$ in $ {}^J_J \cmdcalYD $ are said to be {\it symmetric} if the braidings
\begin{eqnarray*}
  c_{V , W} : V \otimes W \cmdarrow^{\simeq} W \otimes V , \quad v \otimes w \mapsto ( v_{-1} \cmdrhu w ) \otimes v_0 ,\\
  c_{W , V} : W \otimes V \cmdarrow^{\simeq} V \otimes W , \quad w \otimes v \mapsto ( w_{-1} \cmdrhu v ) \otimes w_0
\end{eqnarray*}
are inverses of each other.
One sees that $V$ includes the largest subobject,
which we denote by $V'$, such that $V'$ and $V$ are symmetric,
We call $V'$ the {\it commutant} of $V$.

By gradings we mean those by non-negative integers.
Let $ V \in {}^J_J \cmdcalYD $.
The tensor algebra $ T ( V ) $ on $ V $,
which is naturally a graded algebra,
%2-3
uniquely turns into a braided graded Hopf algebra in $ {}^J_J \cmdcalYD $ in which all elements of $V$ are supposed to be primitives.
Since this $ T ( V ) $ is pointed irreducible as a coalgebra,
its braided bi-ideal is necessarily a braided Hopf ideal.
A quotient algebra $ T ( V ) / \cmdfraka $ is called a {\it pre-Nichols algebra} \cite{M1} of $V$,
if $ \cmdfraka $ is a homogeneous braided Hopf ideal of $ T ( V ) $ such that $ \cmdfraka \cap V = 0 $,
or in other words,
$ V \subset T ( V ) / \cmdfraka $.
This is characterized as such a braided graded Hopf algebra in $ {}^J_J \cmdcalYD $ that includes $V$ as 
a subobject consisting of degree 1 primitives,
%2-4
and is generated by $V$.
A pre-Nichols albegra $ T ( V ) / \cmdfraka $ is called the {\it Nichols algebra} (see \cite{AS1}, \cite{AS2}) of $V$,
if $ \cmdfraka $ is the largest possible,
or equivalently if $V$ precisely equals the space $ P ( T ( V ) / \cmdfraka ) $ of all primitives in $ T ( V ) / \cmdfraka $.
The Nichols algebra of $V$ is uniquely determined by $V$, and is denoted by $ \cmdfrakB ( V ) $.
If two objects $V$, $W$ in $ {}^J_J \cmdcalYD $ are symmetric to each other,
then we have
\begin{equation} \label{eq2.3}
  \cmdfrakB ( V \oplus W ) = \cmdfrakB ( V ) \cmdbarotimes \cmdfrakB ( W ),
\end{equation}
where $ \cmdbarotimes $ denotes the braided tensor product in $ {}^J_J \cmdcalYD $.
This equation follows since the relation
\begin{equation*}
  w v = ( w_{-1} \cmdrhu v ) w_0 \quad ( v \in V , w \in W )
\end{equation*}
holds in $\cmdfrakB ( V \oplus W )$; see \cite[Proposition 2.1]{AS1}, for example.

Given a braided (graded) Hopf algebra,
%2-5
e.g., a (pre-)Nichols algebra,
$R$ in $ {}^J_J \cmdcalYD $,
the bosonization (or the biproduct construction) due to Radford \cite{R1} gives rise to an ordinary (graded) Hopf algebra,
which we denote by
\begin{equation} \label{eq2.4}
  R \cmddotrtimes J .
\end{equation}

We are mostly intersted in the special case when $J$ is the group algebra $k G $ of an (abelian, in most parts) group $G$.
In this case we will write $ {}^G_G \cmdcalYD $ for $ {}^{k G}_{k G} \cmdcalYD $, and $ R \cmddotrtimes G $ for $ R \cmddotrtimes k G $.

\subsection{}

Let $G$ be an abelian group, and let $ \widehat{G} = \cmdHom ( G , k^{\times} ) $ denote the dual group.
%2-6
Given arbitrary elements $ \chi \in \widehat{G} $, $ g \in G $,
we can construct a one-dimensional object $ k v \in {}^G_G \cmdcalYD $ by defining
\begin{equation*}
  h \cmdrhu v = \chi ( h ) v \quad ( h \in G ) , \quad \rho ( v ) = g \otimes v .
\end{equation*}
We will indicate these structures so as 
\begin{equation} \label{eq2.5}
  k v = ( k v ; \chi , g ) .
\end{equation}
Every one-dimensional object in $ {}^G_G \cmdcalYD $ is of this form.
An object $V$ in $ {}^G_G \cmdcalYD $ is said to be {\it of diagonal type} \cite{AS1},
if it is a (direct) sum of one-dimensional objects.
If $k$ is an algebraically closed field of characteristic zero,
and $G$ is finite (abelian),
then every object in $ {}^G_G \cmdcalYD $ is of diagonal type.

%2-7
Suppose that $ V ( \in {}^G_G \cmdcalYD ) $ is of diagonal type, or explicitly
\begin{equation*}
  V = \bigoplus_{i \in \cmdcalI} ( k v_i ; \chi_i, g_i )
\end{equation*}
for some $ \chi_i \in \widehat{G} , g_i \in G $; the index set $ \cmdcalI $ may be infinite.
Since the associated braiding
\begin{equation*}
  c_{i , j} : k v_i \otimes k v_j \cmdarrow^{\simeq} k v_j \otimes k v_i
\end{equation*}
is the scalar multiplication by $ \chi_j ( g_i ) $,
the commutant $ V' $ of $V$ is given by
\begin{equation} \label{eq2.6}
  V' = \bigoplus_{i \in \cmdcalI'} k v_i ,
\end{equation} 
where
\begin{equation} \label{eq2.7}
  \cmdcalI' := \{ i \in \cmdcalI \mid \chi_i ( g_j ) \chi_j ( g_i ) = 1 \ \mbox{ for all } j \in \cmdcalI \}.
\end{equation}
We have
\begin{equation*}
  \cmdfrakB ( V ) = \cmdbarbigotimes_{i \in \cmdcalI'} \cmdfrakB ( k v_i ) \cmdbarotimes \cmdfrakB ( \bigoplus_{i \notin \cmdcalI'} k v_i ).
\end{equation*}
%2-8

\begin{remark} \label{2.1}
  If $ i \in \cmdcalI' $, then $ \chi_i ( g_i )^2 = 1 $.
  It follows that if the characteristic $ \cmdch k $ of $k$ is zero,
  the order $ \left| G \right| $ of $G$ is odd and $ \cmdfrakB ( V ) $ is finite-dimensional,
  then $ V' = 0 $.
  For one sees that if $ ( k v ; \chi, g ) \in {}^G_G \cmdcalYD $ satisfies $ \chi ( g )^2 = 1 $,
  then
  \begin{equation*}
    \cmdfrakB ( k v ) =
    \begin{cases}
      k [ X ] & \quad \mbox{ if } \ \chi ( g ) = 1 \ \mbox{ and } \ \cmdch k = 0,\\
      k [ X ] / ( X^p ) & \quad \mbox{ if } \ \chi ( g ) = 1 \ \mbox{ and } \ \cmdch k = p > 0 ,\\
      k [ X ] / ( X^2 ) & \quad \mbox{ if } \ \chi ( g ) = -1 \ne 1.
    \end{cases}
  \end{equation*}
\end{remark}

\subsection{}
Let $A$ be a pointed Hopf algebra,
and let $ G = G ( A ) $ denote the (not necessarily abelian) group of all grouplikes in $A$.
We let $ \cmdgr A $ denote the graded Hopf algebra associated to the coradical filteration $ k G = A_0 \subset A_1 \subset \cdots $ in $A$.
%2-9
We have uniquely a braided,
strictly graded Hopf algebra $R$ in $ {}^G_G \cmdcalYD $ whose bosonization equals $ \cmdgr A $,
or in notation
\begin{equation*}
  R \cmddotrtimes G = \cmdgr A.
\end{equation*}
Especially, $ R ( 0 ) = k $, and $ R ( 1 ) $ equals the space $ P ( R ) $ of all primitives in $R$.
If $R$ is generated by $ P ( R ) $,
or equivalently if $A$ is generated by skew primitives,
then $ R = \cmdfrakB ( P ( R ) ) $.
We define
\begin{equation} \label{eq2.8}
  \cmdYD ( A ) := P ( R ) \ ( = R ( 1 ) ),
\end{equation}
which we call the {\it infinitesimal Yetter-Drinfeld module} associated to $A$.
We see easily the following.

%2-10
\begin{lemma} \label{2.2}
  $ \cmdYD ( B ) \subset \cmdYD ( A ) $ if $B$ is a Hopf subalgebra of $A$.
\end{lemma}

We will study Hopf ideals of a certain pointed Hopf algebra $ \cmdcalD $ which will be constructed in the following section.
Recall that a bi-ideal of a pointed Hopf algebra is necessarily a Hopf ideal.

%3-1
\section{The generalized quantum double}

Let $F$, $G$ be (possibly infinite) abelian groups.
We will write $ f g $ for the element $ ( f , g ) $ in $ F \times G $.
Suppose that we are given a bimultiplicative map
\begin{equation*}
  \tau_0 : F \times G \rightarrow k^{\times} ( = k \setminus 0 ).
\end{equation*}
For $ f \in F $, $ g \in G $, define $ \widehat{g} \in \widehat{F} $, $ \widehat{f} \in \widehat{G} $ by
\begin{equation*}
  \widehat{g} (f) = \tau_0 (f , g) = \widehat{f} (g).
\end{equation*}
We will also write $ \widehat{g} \widehat{f} $ for the element $ ( \widehat{g} , \widehat{f} ) $ in $ \widehat{F} \times \widehat{G} ( = (F \times G)^{\wedge} ) $.
Let $ \cmdcalI ( \ne \phi ) $ be a (possibly infinite) index set.
Choose arbitrarily elements $ f_i \in F, g_i \in G $ for each $ i \in \cmdcalI $,
and define Yetter-Drinfeld modules,
%3-2
\begin{equation*}
  V = \bigoplus_{ i \in \cmdcalI } ( k v_i ; \widehat{g}_i , f_i ) \ \mbox{ in } {}^F_F \cmdcalYD, \quad W = \bigoplus_{ i \in \cmdcalI } ( k w_i ; \widehat{f}^{-1}_i , g_i ) \ \mbox{ in } {}^G_G \cmdcalYD.\\
\end{equation*}
Suppose that we are given pre-Nichols algebras
\begin{equation*}
  R~ ( \in {}^F_F \cmdcalYD ) , \quad S~ ( \in {}^G_G \cmdcalYD )
\end{equation*}
of $V$, $W$, respectively.
Let
\begin{equation} \label{eq3.1}
  B = R \cmddotrtimes F, \quad H = S \cmddotrtimes G
\end{equation}
denote their bosonizations.
Thus, $B$, $H$ are pointed Hopf algebras with $ G ( B ) = F $, $ G ( H ) = G $, in which the coproducts $ \Delta ( v_i ) , \Delta ( w_i ) \ ( i \in \cmdcalI ) $ are given by
\begin{equation*}
  \Delta ( v_i ) = v_i \otimes 1 + f_i \otimes v_i , \quad \Delta ( w_i ) = w_i \otimes 1 + g_i \otimes w_i .
\end{equation*}

\begin{proposition} \label{3.1}
  $ \tau_0 $ uniquely extends to such a skew pairing $ \tau : B \otimes H \rightarrow k $ that satisfies
  \begin{equation} \label{eq3.2}
    \tau ( v_i , w_j ) = \delta_{i j} \quad ( i , j \in \cmdcalI ),
  \end{equation}
  \begin{equation} \label{eq3.3}
    \tau ( f , w ) = 0 = \tau ( v , g ) \quad ( f \in F , g \in G , v \in V , w \in W ).
  \end{equation}
\end{proposition}

%3-3
We continue to denote this skew pairing by $\tau$.
By definition of skew pairings,
$ \tau $ satisfies
\begin{eqnarray*}
  \tau ( bc , h ) = \tau ( b , h_1 ) \tau ( c , h_2 ),\\
  \tau ( b , hl ) = \tau ( b_1 , l ) \tau ( b_2 , h ),\\
  \tau ( b , 1 ) = \varepsilon ( b ) , \quad \tau ( 1 , h ) = \varepsilon ( h ),
\end{eqnarray*}
where $ b , c \in B $, $ h , l \in H $.

\begin{proof}[Proof of Proposition \ref{3.1}.]
  In the situation of \cite[Theorem 5.3]{M1},
  we can suppose $J$, $K$ to be our $ k F $, $ k G $.
  We can also suppose the $ \lambda $ in \cite{M1} to be such that $ \lambda ( w_j , v_i ) = \delta_{ i j } $.
  In this special case the first half of \cite[Theorem 5.3]{M1} is precisely the proposition above.
  See also \cite[Theorem 8.3]{RS1}, \cite[Lemma 3.1]{RS2}.
\end{proof}

%3-4
\begin{proposition} \label{3.2}
  The restriction $ \tau \! \mid_{R \otimes S} : R \otimes S \rightarrow k $ is non-degenerate if and only if $ R , S $ are both Nichols.
\end{proposition}

\begin{proof}
  This follows by \cite[Proposition 2.6(1)]{M2}, since our $ \tau \! \mid _{V \otimes W} $ is non-degenerate.
\end{proof}

By \cite[Proposition 1.5]{DT}, the linear map $ \sigma : ( B \otimes H ) \otimes ( B \otimes H ) \rightarrow k $ defined by
\begin{equation*}
  \tau ( b \otimes h , c \otimes l ) = \varepsilon ( b ) \tau ( c , h ) \varepsilon ( l )
\end{equation*}
is a $2$-cocyle for the Hopf algebra $ B \otimes H $ of tensor product.
We let
\begin{equation} \label{eq3.4}
  \cmdcalD = B \cmdbicross_{\tau} H
\end{equation}
denote the cocyle deformation $ ( B \otimes H )^{ \sigma } $ of $ B \otimes H $ by $ \tau $; see \cite[Sect. 2]{DT}.

%3-5
\begin{proposition} \label{3.3}
  This $ \cmdcalD $ is the Hopf algebra which is defined on $ H \otimes B $,
  uniquely so that it includes $ B = B \otimes k $, $ H = k \otimes H $ as Hopf subalgebras,
  and satisfies the relations
  \begin{eqnarray}
    g f &=& f g, \label{eq3.5} \\
    g v_i &=& \tau_0 ( f_i , g ) v_i g, \label{eq3.6} \\
    w_i f &=& \tau_0 ( f , g_i ) f w_i, \label{eq3.7} \\
    w_j v_i &=& \tau_0 ( f_i , g_j ) v_i w_j + \delta_{i j} ( f_i g_i - 1 ), \label{eq3.8}
  \end{eqnarray}
  where $ f \in F $, $ g \in G $, $ i \in \cmdcalI $.
\end{proposition}

\begin{proof}
  In the specialized situation described in the proof of Proposition \ref{3.1},
  the second half of \cite[Theorem 5.3]{M1} is precisely the proposition to be proven.
\end{proof}

%3-6
We call $ \cmdcalD $ the {\it generalized quantum double} associated to $ \tau $.

\begin{remark} \label{3.4}
  Suppose that $ \tau_0 $ is non-degenerate,
  and that $R$, $S$ are both Nichols, so that $ R = \cmdfrakB ( V ) $, $ S = \cmdfrakB ( W ) $.
  By \cite[Corollary 2.4, Proposition 2.6(1)]{M2},
  $ \tau $ is then non-degenerate.
  Suppose in addition that $F$ or $G$ is
  (then both are) finite,
  and that $ \cmdfrakB ( V ) $ or $ \cmdfrakB ( W ) $ is
  (then both are) finite-dimensional.
  Then, $ H = \cmdfrakB ( W ) \cmddotrtimes G $ is a finite-dimensional Hopf algebra,
  and $ \cmdcalD $ is (isomorphic to) the quantum double of $H$ as Drinfeld originally defined.
  We denote this by $ D ( H ) $.
\end{remark}

%3-7
Let us return to the general situation before Remark \ref{3.4}.
As is seen from the proof of \cite[Theorem 5.3]{M1},
one can construct a pointed (graded) Hopf algebra,
say $ \cmdcalE $,
just as the Hopf algebra $ \cmdcalD $ as described by Proposition \ref{3.3},
but replacing (\ref{eq3.8}) by
\begin{equation*}
  w_j v_i = \tau_0 ( f_i , g_j ) v_i w_j \quad ( i , j \in \cmdcalI ).
\end{equation*}
We see from that proof, moreover, the following.

\begin{proposition} \label{3.5}
  \begin{enumerate}
% 5a
    \item $V$, $W$ turn into objects in $ {}^{ F \times G }_{ F \times G } \cmdcalYD $ of diagonal type so that
    \begin{equation*}
      V = \bigoplus_{i \in \cmdcalI} ( k v_i ; \widehat{g}_i \widehat{f}_i , f_i ) , \quad W = \bigoplus_{i \in \cmdcalI} ( k w_i ; \widehat{g}^{-1}_i \widehat{f}^{-1}_i , g_i ).
    \end{equation*}
    The graded algebras $R$, $S$ turn into pre-Nichols algebras of $V$, $W$ in  $ {}^{ F \times G }_{ F \times G } \cmdcalYD $, respectively.
    \item The braided tensor product $ R \cmdbarotimes S $ in $ {}^{ F \times G }_{ F \times G } \cmdcalYD $ is a pre-Nichols algebra of $ V \oplus W $;
    it is Nichols if the original $R$, $S$ are both Nichols.
    The bosonization $ ( R \cmdbarotimes S ) \cmddotrtimes ( F \times G ) $ is precisely $ \cmdcalE $.
%3-8
    \item $ \cmdcalD $ is a cocycle deformation of $ \cmdcalE $, too,
    whence $ \cmdcalD = \cmdcalE $ as coalgebras.
    \item $ \cmdcalD $ is a pointed Hopf algebra with $ G ( \cmdcalD ) = F \times G $,
    and the associated infinitesimal Yetter-Drinfeld module $ \cmdYD ( \cmdcalD ) $ includes $ V \oplus W $
    as a subobject.
    We have $ \cmdYD ( \cmdcalD ) = V \oplus W $,
    if $R$, $S$ are both Nichols.
  \end{enumerate}
\end{proposition}

It what follows we keep the notation
\begin{equation*}
  F , \ G , \ \tau_0 , \ V , \ W , \ R , \ S , \ \tau , \ \cmdcalD
\end{equation*}
as above.
We also keep $C$ denoting the subgroup of $ F \times G $ defined by
\begin{equation*}
  C = \{ f g \mid \tau_0 ( f , g_i ) \tau_0 ( f_i , g ) = 1 \ ( i \in \cmdcalI ) \}.
\end{equation*}
%3-9
Notice from (\ref{eq3.6}), (\ref{eq3.7}) that $C$ consists of all grouplikes that are central in $ \cmdcalD $.

%4-1
\section{Tensor product decomposition of $ \cmdcalD $}

Let
\begin{equation*}
  P = \left< f_i g_i \mid i \in \cmdcalI \right>
\end{equation*}
denote the subgroup of $ F \times G $ generated by all $ f_i g_i ( i \in \cmdcalI ) $.
Note that all $ c - 1 ( c \in C ) $,
being (central) skew primitives,
generate a Hopf ideal in $\cmdcalD$.
Let
\begin{equation*}
  \cmdcalD_c = \cmdcalD / ( c - 1 \mid c \in C )
\end{equation*}
denote the quotient Hopf algebra of $ \cmdcalD $ by that Hopf ideal.

\begin{proposition} \label{4.1}
  Consider the following three conditions.
  \begin{itemize}
    \item[(a)] The natural group map $ C \rightarrow F \times G / P $ is a split mono.
    \item[(b)] The inclusion $ k C \hookrightarrow \cmdcalD $ splits as a Hopf algebra map.
    \item[(c)] There is an algebra isomorphism $ \cmdcalD \simeq k C \otimes \cmdcalD_c $.
  \end{itemize}
  Then we have $ (a) \Rightarrow (b) \Rightarrow (c) $.
\end{proposition}

\begin{proof}
  $ (a) \Rightarrow (b) $.
  Assume $(a)$, or equivalently that $ C \hookrightarrow F \times G $ has a retraction $ \pi : F \times G \rightarrow C $ such that $ P \subset \cmdKer \pi $.
  As is seen from Proposition \ref{3.3}, $ \pi $ can uniquely extend to a Hopf algebra retraction $ \pi : \cmdcalD \rightarrow k C $ so that $ \pi ( v_i ) = 0 = \pi ( w_i ) $ for all $ i \in \cmdcalI $.

  $ (b) \Rightarrow (c) $.
  Let $ \pi : \cmdcalD \rightarrow k C $ be a retraction as assumed by $ ( b ) $.
  By Radford \cite[Theorem 3]{R1}, a (right $ \cmdcalD_c $-comodule) algebra isomorphism $ \cmdcalD \simeq k C \otimes \cmdcalD_c $ is given by $ a \mapsto \pi ( a_1 ) \otimes \bar{a}_2 $.
\end{proof}

We are going to prove in our context two results due to Leonid Krop \cite{Kr},
who worked in the situation of Remark \ref{3.4}; our formulation is generalized.

%4-3
\begin{theorem}[Krop] \label{4.2}
  Assume that
  \begin{equation} \label{eq4.1}
    F \mbox{ or } G \mbox{ is finite, and }
  \end{equation}
  \begin{equation} \label{eq4.2}
    \tau_0 : F \times G \rightarrow k^{ \times } \mbox{ is non-degenerate.}
  \end{equation}
  (Note then $ \left| F \right| = \left| G \right| < \infty $.)
  \begin{enumerate}
    \item We have
    \begin{equation*}
      \left| C \right| \left| P \right| = \left| F \right| \left| G \right |,
    \end{equation*}
    and the condition (a) above is equivalent to each of the following:
    \begin{itemize}
      \item [(a$'$)] $ C \rightarrow F \times G / P $ is monic/epic/isomorphic;
      \item [(a$''$)] The product map $ C \times P \rightarrow F \times G $ is monic/epic/isomorphic.
    \end{itemize}
    \item Assume in addition that
    \begin{equation} \label{eq4.3}
      \mbox{ for each } i \in \cmdcalI , f_i \ne 1 \mbox{ or } g_i \ne 1 \ ( \mbox{or in short, } f_i g_i \ne 1 ).
    \end{equation}
    Then the conditions (a), (a$'$), (a$''$), (b), (c) above are equivalent to each other.
  \end{enumerate}
\end{theorem}

\begin{proof}
  (1) By (\ref{eq4.1}), (\ref{eq4.2}),
  $ \tau_0 $ induces an isomorphism $ \displaystyle F \times G \cmdarrow^{\simeq} ( F \times G )^{\wedge} \ $.
  %4-4
  Compose this with the restriction map $ ( F \times G )^{\wedge} \rightarrow \widehat{P} $,
  which is epic since by (\ref{eq4.1}), (\ref{eq4.2}),
  $k$ contains a root of $1$ whose order equals $ \exp F ( = \exp G ) $,
  the exponent.
  The desired result follows since the kernel of the composite is $C$.

  (2) Assume (c).
  We wish to prove that the product map $ C \times P \rightarrow F \times G $ is epic.
  %4-5
  Given an algebra $A$, let $ n( A ) $ denote the number of all algebra maps $ A \rightarrow k $.
  By (c),
  \begin{equation} \label{eq4.4}
    n ( \cmdcalD ) = n ( k C ) n ( \cmdcalD_c ) = \left| C \right| n ( \cmdcalD_c ).
  \end{equation}
  Let $ \varphi : \cmdcalD \rightarrow k $ be an algebra map.
  By (\ref{eq3.6}), (\ref{eq3.7}), (\ref{eq3.8}), (\ref{eq4.2}) and (\ref{eq4.3}),
  we see
  \begin{equation*}
    \varphi ( v_i ) = 0 = \varphi ( w_i ) \quad ( i \in \cmdcalI ) , \quad \varphi ( P ) = \{ 1 \},
  \end{equation*}
  which implies $ n ( \cmdcalD ) = \left| F \times G / P \right| = \left| C \right| $.
  By (\ref{eq4.4}),
  $ \varphi $ must coincide with the counit if $ \varphi ( C ) = \{ 1 \} $.
  We have shown that a group map $ F \times G \rightarrow k^{ \times } $ is trivial if it vanishes on $ C P $.
  This implies $ C P = F \times G $,
  as desired.
\end{proof}

% 4A-1
Suppose that $ \cmdcalI $ is finite, and $ \cmdbbA = ( a_{i j} )_{ i,j \in \cmdcalI } $ is a Cartan matrix of finite type, which is symmetrized by a diagonal matrix $ \cmdbbD = \cmddiag ( \cdots d_i \cdots )_{ i \in \cmdcalI } $,
where $ d_i \in \{ 1 ,2, 3 \} $, in the standard way.
Let $ l > 1 $ be an odd integer.
We assume that $l$ is not a multiple of $3$ if $\cmdbbA$ contains an irreducible component of type $ G_2 $.
Suppose that $k$ contains a primitive $ l^{\cmdth} $ root $q$ of $1$.
Suppose $ F = G $, and that $G$ is the finite abelian group which is generated by $ K_i \ ( i \in \cmdcalI ) $, and is defined by $ K^l_i = 1 \ ( i \in \cmdcalI ) $.
Choose $f_i$, $g_i$, $\tau_0$ so that
\begin{equation*}
  f_i = g_i = K_i , \quad \tau_0 ( K_i , K_j ) = q^{-d_ia_{i j}} \quad ( i , j \in \cmdcalI ).
\end{equation*}
%4A-2
Let $T$ be the subgroup of $ G \times G $ defined by
\begin{equation*}
  T = \{ ( g , g^{-1} ) \mid g \in G \}~( = \left< ( K_i , K^{-1}_i ) \mid i \in \cmdcalI \right> ).
\end{equation*}
Note $ T \subset C $.
Let
\begin{equation*}
  \cmdcalD_t = \cmdcalD / ( K_i \otimes 1 - 1 \otimes K_i \mid i \in \cmdcalI )
\end{equation*}
denote the quotient Hopf algebra of $ \cmdcalD $ corresponding to $T$.
If $ \cmdch k = 0 $ and $R$, $S$ are both Nichols, then $\cmdcalD_t$ coincides with that quotient Hopf algebra of the Frobenius-Lusztig kernel $ u_q $ (see \cite{L}, \cite{Mu})
which is obtained by replacing the relations $ K^{2l}_i = 1 \ ( i \in \cmdcalI ) $ for $ u_q $ with $ K^l_i = 1 \ ( i \in \cmdcalI )$.

\begin{corollary}[Krop] \label{4.3}
  In the situation above, the following (i)-(iv) are equivalent to each other.
  \begin{itemize}
% 4A-3
    \item [(i)] $T = C$;
    \item [(ii)] $C \cap P = \{ 1 \}$;
    \item [(iii)] $\tau_0$ is non-degenerate;
    \item [(iv)] $l$ is prime to $ \cmddet \cmdbbA $ (by the known values of the determinants,
    this condition is equivalent to that $ \cmdbbA $ does not contain any irreducible component of type $ A_n $ of rank $ n $ such that $ n + 1 $ is not prime to $l$,
    and if $l$ is a multiple of $3$, $\cmdbbA$ does not contain an irreducible component of type $ E_6 $ either).
  \end{itemize}
  If these equivalent conditions are satisfied,
  then $ \cmdcalD_t = \cmdcalD_c $ and the following (v), (vi) hold true.
% 4A-4
  \begin{itemize}
    \item [(v)] $ kC \hookrightarrow \cmdcalD $ splits as a Hopf algebra map.
    \item [(vi)] $\cmdcalD$ is isomorphic to $ kC \otimes \cmdcalD_c $, as a right (or left) $\cmdcalD_c$-comodule algebra.
  \end{itemize}
\end{corollary}

\begin{proof}
  The equivalence follows since one sees that each of (i)-(iv) is equivalent to that for arbitrary integers $n_i$,
  \begin{equation*}
    \sum_i n_i d_i a_{ij} \equiv 0 \ \cmdmod \ l \quad ( \forall j \in \cmdcalI ) \
  \end{equation*}
  implies $ n_i \equiv 0 \ \cmdmod \ l \ ( \forall i \in \cmdcalI ) $.
  As for (iv), note from the assumption on $l$ that each $d_i$ is prime to $l$.
  By the proof of Theorem \ref{4.2}(2), (ii) and (iii) imply (v), (vi).
\end{proof}

%5-1
\section{Skinny and thin Hopf ideals of $ \cmdcalD $}

Throughout this section we assume that $R$, $S$ are both Nichols,
so that $ R = \cmdfrakB ( V ) $, $ S = \cmdfrakB ( W ) $.

Given a subgroup $T$ of $C$, the ideal $ ( t - 1 \mid t \in T ) $ generated by all $ t - 1 \ ( t \in  T ) $ is a Hopf ideal.

\begin{definition} \label{5.1}
  A Hopf ideal $ \cmdfraka $ of $ \cmdcalD $ is said to be {\rm skinny},
  if $ \cmdfraka = ( t - 1 \mid t \in T ) $ for some subgroup $ T \subset C $.
  We say that $ \cmdfraka $ is {\rm thin}, if
  \begin{equation*}
    V \cap \cmdfraka = 0 = W \cap \cmdfraka.
  \end{equation*}
\end{definition}

Let $ T \subset C $ be a subgroup, and set
\begin{equation*}
  \Gamma = F \times G / T.
\end{equation*}
%5-2
Let
\begin{equation*}
  \cmdcalI_T = \{ i \in \cmdcalI \mid f_i g_i \in T \},
\end{equation*}
and set
\begin{equation*}
  V_T = \bigoplus_{ i \in \cmdcalI_T } k v_i , \quad W_T = \bigoplus_{ i \in \cmdcalI_T } k w_i.
\end{equation*}
Note that if $ T = C $, then
\begin{equation*}
  \cmdcalI_C = \{ i \in \cmdcalI \mid \tau_0 ( f_i , g_j ) \tau_0 ( f_j , g_i ) = 1~( j \in \cmdcalI ) \},\\
\end{equation*}
\begin{equation*}
  V_C = V' , \quad W_C = W' ,
\end{equation*}
where $ V' $, $ W' $ denote the commutants; see (\ref{eq2.6}).
Therefore,
\begin{equation} \label{eq5.1}
  V_T \subset V' , \quad W_T \subset W'.
\end{equation}
We can naturally regard $ \widehat{g}_i \widehat{f}_i $ as an element in $ \widehat{\Gamma} $, and $ V $, $ W $ as diagonal type objects in $ {}^{ \Gamma }_{ \Gamma } \cmdcalYD $.
Choose arbitrarily
\begin{equation*}
  Z \subset V_T, \quad \zeta : Z \hookrightarrow W_T
\end{equation*}
%5-3
so that $ Z \subset V_T $ is a subobject,
and $ \zeta : Z \hookrightarrow W_T $ is a mono,
both in $ {}^\Gamma_\Gamma \cmdcalYD $.
Let
\begin{equation} \label{eq5.2}
  \cmdfraka = \cmdfraka ( T, Z, \zeta )
\end{equation}
denote the ideal of $ \cmdcalD $ generated by all $ t - 1 $, $ z - \zeta ( z ) $,
where $ t \in T $, $ z \in Z $.

\begin{proposition} \label{5.2}
  This $ \cmdfraka $ is a thin Hopf ideal of $ \cmdcalD $.
  In particular,
  a skinny Hopf ideal of $ \cmdcalD $ is thin.
\end{proposition}

\noindent {\it Proof.} The second statement follows from the first, by choosing $ Z , \zeta $ so as to be zero.

To prove the first statement,
set
\begin{equation*}
  V^c_T = \bigoplus_{ i \notin \cmdcalI_T } k v_i , \quad W^c_T = \bigoplus_{ i \notin \cmdcalI_T } k w_i .
\end{equation*}
We can choose compliments $X$, $Y$ of $Z$, $ \zeta ( Z ) $ in $ V_T $, $ W_T $, respectively, so that
% 5-4
\begin{equation*}
  V = V^c_T \oplus X \oplus Z , \quad W = \zeta ( Z ) \oplus Y \oplus W^c_T \quad \mbox{ in } {}^{ \Gamma }_{ \Gamma } \cmdcalYD.
\end{equation*}
By (\ref{eq5.1}),
the six direct summands above are pairwise symmetric.
One sees that
\begin{equation*}
  \alpha : Z \oplus \zeta ( Z ) \rightarrow \zeta ( Z ) , \quad \alpha ( z , \zeta ( z' ) ) = \zeta ( z + z' )
\end{equation*}
is a split epi in $ {}^{ \Gamma }_{ \Gamma } \cmdcalYD $,
whose kernel $ \cmdKer \alpha $ consists of $ z - \zeta ( z ) \ ( z \in Z ) $.
Therefore, $ \alpha $ induces a split epi
\begin{equation*}
  \tilde{\alpha} : \cmdfrakB ( Z ) \cmdbarotimes \cmdfrakB ( \zeta ( Z ) ) \rightarrow \cmdfrakB ( \zeta ( Z ) )
\end{equation*}
of braided graded Hopf algebras,
whose kernel $ \cmdKer \tilde{\alpha} $ is generated by $ \cmdKer \alpha $.

Set
\begin{equation} \label{eq5.3}
  U := \frac{ V \oplus W }{ \cmdKer \alpha }~( \simeq V^c_T \oplus X \oplus \zeta ( Z ) \oplus Y \oplus W^c_T ).
\end{equation}
%5-5
Then,
\begin{equation*}
  \cmdfrakB ( U ) = \cmdfrakB ( V^c_T ) \cmdbarotimes \cmdfrakB ( X ) \cmdbarotimes \cmdfrakB ( \zeta ( Z ) ) \cmdbarotimes \cmdfrakB ( Y ) \cmdbarotimes \cmdfrakB ( W^c_T ).
\end{equation*}
Let $ \bar{f_i} $, $ \bar{g_i} $ denote the natural image of $ f_i $, $ g_i $ in $ \Gamma $.
To continue the proof we need the next lemma, which follows by \cite[Theorem 4.10]{M1}.

\begin{lemma} \label{5.3}
  The bosonization $ \cmdfrakB ( U ) \cmddotrtimes \Gamma $ is deformed by a $2$-cocycle into a Hopf algebra so that for all $i , j \in \cmdcalI \setminus \cmdcalI_T $,
  the original relations
  \begin{equation*}
    w_j v_i = \tau ( f_i , g_j ) v_i w_j
  \end{equation*}
  are deformed to
  \begin{equation*}
    w_j v_i = \tau_0 ( f_i , g_j ) v_i w_j + \delta_{i j} ( \bar{f_i} \bar{g_i} - 1 ).
  \end{equation*}
\end{lemma}

We denote this deformed Hopf algebra by
\begin{equation} \label{eq5.4}
  \cmdcalA = \cmdcalA ( T , Z , \zeta ).
\end{equation}
Note that $ \cmdcalA = \cmdfrakB ( U ) \cmddotrtimes \Gamma $ as coalgebras,
and the first two components of the thus coradically graded coalgebra $ \cmdcalA $ are given by
%5-6
\begin{equation} \label{eq5.5}
  \cmdcalA ( 0 ) = k \Gamma, \quad \cmdcalA ( 1 ) = U \otimes k \Gamma .
\end{equation}

\begin{proof}[Proof of Proposition \ref{5.2} (continued)]
  Since $ \bar{f_i} \bar{g_i} - 1 = 0 $ for $ i \in \cmdcalI_T $, one sees from Proposition \ref{3.3} that the natural projections $ F \times G \rightarrow \Gamma $,
  $ V \oplus W \rightarrow U $ induce a Hopf algebra epi $ \pi : \cmdcalD \rightarrow \cmdcalA $.
  Since $ \pi $ is the tensor product of $ \tilde{\alpha} $,
  the projection $ k ( F \times G ) \rightarrow k \Gamma $ and the identity maps of four Nichols algebras,
  we see that $ \cmdKer \pi $ is generated by $ t - 1 \ ( t \in T ) $ and $ \cmdKer \tilde{\alpha} $.
  This implies $ \cmdKer \pi = \cmdfraka $,
  whence $ \cmdcalD / \cmdfraka \simeq \cmdcalA $.
  (Therefore, $ \cmdcalA $ does not depend, up to isomorphism, on choice of $ X $, $ Y $.)
  %5-7
  Since $ \cmdcalD / \cmdfraka \ ( \simeq \cmdcalA ) $ includes $ V $, $ W $ (even $ \cmdfrakB ( V ) $, $ \cmdfrakB ( W ) $),
  $ \cmdfraka $ is a thin Hopf ideal.
\end{proof}

\begin{theorem} \label{5.4}
  Assume that $ R = \cmdfrakB ( V ) $, $ S = \cmdfrakB ( W ) $, and in addition,
  \begin{equation} \label{eq5.6}
    \mbox{ for each } i \in \cmdcalI, \widehat{f}_i \ne 1 \mbox{ or } \widehat{g}_i \ne 1~ ( \mbox{or in short, } \widehat{g}_i \widehat{f}_i \ne 1 ).
  \end{equation}
  Then a thin Hopf ideal of $ \cmdcalD $ is necessarily of the form $ \cmdfraka ( T, Z , \zeta ) $ for some $ T , Z , \zeta $.
\end{theorem}

\begin{proof}
  Let $ \cmdfraka \subset \cmdcalD $ be a Hopf ideal.
  Set $ A = \cmdcalD / \cmdfraka $, and let $ \pi : \cmdcalD \rightarrow A $ denote the quotient map.
  Let
  \begin{equation*}
    T = \{ t \in F \times G \mid \pi ( t ) = 1 \},
  \end{equation*}
  and set $ \Gamma = F \times G / T $.
  Then $ \Gamma = G ( A ) $.
  %5-7
  Assume that $ \cmdfraka $ is thin.
  Then $V$, $W$ are embedded into $A$ via $\pi$.
  From the relations
  \begin{equation*}
    ( f g ) v_i = \tau_0 ( f_i , g ) \tau_0 ( f , g_i ) v_i ( f g ) \quad ( f \in F , g \in G , i \in \cmdcalI )
  \end{equation*}
  in $ \cmdcalD $, we see $ T \subset C $.
  Therefore, $V$, $W$ can be naturally regarded as objects in $ {}^{ \Gamma }_{ \Gamma } \cmdcalYD $, as before.
  Set the sum
  \begin{equation} \label{eq5.7}
    U := V + W \mbox{ in } A.
  \end{equation}
  This is stable under $ \Gamma $-conjugation,
  and does not contain any non-zero $ \Gamma $-invariant by(\ref{eq5.6}).
  Therefore, $ U \cap k \Gamma = 0 $.
  This implies that $U$ is naturally embedded into $ \cmdgr A ( 1 ) $,
  or more specifically into $ \cmdYD ( A ) $, which is the neutral component of the right $ k \Gamma ( = \cmdgr A ( 0 ) ) $-comodule $ \cmdgr A ( 1 ) $.
  %5-8
  Note that by this embedding $ U \hookrightarrow YD ( A ) $, $U$ is a subobject in $ {}^{ \Gamma }_{ \Gamma } \cmdcalYD $,
  and is indeed a sum of the two objects $V$, $W$ in $ {}^{ \Gamma }_{ \Gamma } \cmdcalYD $.
  For each $ \gamma \in \Gamma $, $ U \gamma $ is embedded into the $ \gamma $-component in $ \cmdgr A ( 1 ) $.
  Hence we have $ \bigoplus_{ \gamma \in \Gamma } U \gamma \hookrightarrow \cmdgr A ( 1 ) $,
  and so
  \begin{equation} \label{eq5.8}
    ( k \oplus U ) \otimes k \Gamma \hookrightarrow A
  \end{equation}
  via the product map.

  Let $Z$ be the pullback $ ( \pi \! \mid_V )^{-1} ( V \cap W ) $ of $ V \cap W ( \subset U ) $ along the embedding $ \pi \! \mid_V $.
  There exists a unique mono $ \zeta : Z \rightarrow W $ in $ {}^\Gamma_\Gamma \cmdcalYD $ such that $ \pi \! \mid_Z = \pi \! \mid_W \circ \ \zeta $.
  We wish to prove
  \begin{equation} \label{eq5.9}
    Z \subset V_T , \quad \zeta ( Z ) \subset W_T.
  \end{equation}
  To prove the first inclusion,
  let $ z = \sum_i c_i v_i \in Z $ with $ c_i \in k $.
  Choose arbitrarily $ j \in \cmdcalI $ such that $ c_j \ne 0 $.
  We may suppose that $z$ spans a simple object,
  so that if $ c_i \ne 0 $,
  then $ \pi ( f_i ) = \pi ( f_j ) $, $ \widehat{g}_i \widehat{f}_i = \widehat{g}_j \widehat{f}_j $.
  Then we have
  \begin{equation*}
    w_j z - \tau_0 ( f_j , g_j ) z w_j = c_j ( f_j g_j -1 )
  \end{equation*}
  in $ \cmdcalD $.
  By applying $ \pi $, we have
  \begin{equation} \label{eq5.10}
    w_j \zeta ( z ) - \tau_0 ( f_j , g_j ) \zeta ( z ) w_j = c_j ( \pi ( f_j g_j ) - 1 ).
  \end{equation}
  %5-10
  But, since (\ref{eq5.8}) implies that the natural map $ \cmdfrakB ( W ) \cmddotrtimes \Gamma  \rightarrow A $ is monic (see \cite[5.3.1 Theorem]{Mo}),
  both sides of (\ref{eq5.10}) must be zero,
  whence $ f_j g_j \in T $; this proves the first inclusion of (\ref{eq5.9}).
  The second follows similarly,
  by using the fact, which follows from (\ref{eq5.8}), that $ \cmdfrakB ( V ) \cmddotrtimes \Gamma \rightarrow A $ is monic.

  From $T$, $Z$, $ \zeta $ above, construct the Hopf algebra $ \cmdcalA = \cmdcalA ( T , Z , \zeta ) $; see (\ref{eq5.4}).
  Note that the $U$ given by (\ref{eq5.3}) and the $U$ given by (\ref{eq5.7}) are now identified canonically.
  We see that $ \pi $ induces a Hopf algebra epi $ \cmdcalA \rightarrow A $ which is identical on $ \Gamma $ and $U$.
  It is indeed an isomorphism by (\ref{eq5.5}), (\ref{eq5.8}).
  %5-11
  This proves the theorem.
\end{proof}

\begin{remark} \label{5.5}
  As is easily seen, we must assume $ R = \cmdfrakB ( V ) $, $ S = \cmdfrakB ( W ) $, to have the conclusion of Theorem \ref{5.4}.
  We must assume (\ref{eq5.6}), too, as will be seen below.
  In fact, we have an example for which there exists an $ i \in \cmdcalI $ such that
  \begin{equation} \label{eq5.11}
    \tau_0 ( f_i , g_j ) = 1 = \tau_0 ( f_j , g_i ) \quad ( j \in \cmdcalI ),
  \end{equation}
  \begin{equation} \label{eq5.12}
    f_i \notin T ,
  \end{equation}
  where $ T = \left< f_i g_i \right> $, the subgroup of $ F \times G $ generated by $ f_i g_i $.
  By (\ref{eq5.11}), $ T \subset C $, but $ \widehat{f}_i = 1 = \widehat{g}_i $, breaking (\ref{eq5.6}).
  Note that $ v_i $ and $ \Gamma := F \times G / T $ generate a central Hopf subalgebra in $ \cmdcalD / ( f_i g_i - 1 ) $.
  %5-12
  Then one sees that
  \begin{equation*}
    \cmdfraka = ( v_i - ( f_i - 1 ) , f_i g_i - 1 )
  \end{equation*}
  is a thin Hopf ideal by (\ref{eq5.12}),
  but cannnot be of the form (\ref{eq5.2}) since $ v_i \in k \Gamma $ in $ \cmdcalD / \cmdfraka $.
\end{remark}

%6-1
\section{All Hopf ideals of $ \cmdcalD $ with connectedness assumption}

Distinct $i$, $j$ in $ \cmdcalI $ are said to be {\it adjacent} to each other if $ k v_i $ and $ k v_j $ (or $ kw_i $ and $ k w_j $) are not symmetric,
or namely if
\begin{equation} \label{eq6.1}
  \tau_0 ( f_i , g_j ) \tau_0 ( f_j , g_i ) \ne 1 .
\end{equation} 

\begin{theorem} \label{6.1}
  Assume that $R$, $S$ are both Nichols, and assume one of the following.
  \begin{itemize}
    \item [(i)] $ \cmdcalI $ consists of only one element, say $1$,
    so that $ \cmdcalI = \{ 1 \} $, and $ \tau ( f_1 , g_1 )^2 \ne 1 $.
    \item [(ii)] $ \left| \cmdcalI \right| > 1 $, and $ \cmdcalI $ is connected in the sense that if $ i \ne j $ in $ \cmdcalI $,
    there exists a finite sequence $ i = i_0 , i_1 , \cdots , i_r = j $ in $ \cmdcalI $ such that for each $ 0 \leq s < r $,
  %6-2
    $ i_s $ and $ i_{s + 1} $ are distinct and adjacent.
  \end{itemize}
  Then every Hopf ideal $ \cmdfraka $ of $ \cmdcalD $
  \begin{itemize}
    \item [(a)] is skinny (see Definition \ref{5.1}), or
    \item [(b)] contains all $ v_i $, $ w_i $, $ f_i g_i - 1 \ ( i \in \cmdcalI ) $.
  \end{itemize}
\end{theorem}

This result together with its proof directly generalizes those of Chin and Musson \cite[Theorem C]{CM}, and of M\"uller \cite[Theorem 4.2]{Mu}; see also \cite[Lemma A.2]{A}.

\begin{proof}
  Let $ \cmdfraka \subset \cmdcalD $ be a Hopf ideal.
  We will work first in Case (ii),
  and then in Case (i).

  Case (ii).
  First, suppose that for a specific $ i \in \cmdcalI $, $ \cmdfraka $ contains $ v_i $, $ w_i $, or $ f_i g_i - 1 $.
  %6-3
  In any case, $ f_i g_i - 1 \in \cmdfraka $ by (\ref{eq3.8}).
  Choose $ j ( \ne i ) $ which is adjacent to $i$.
  Then the relation
  \begin{equation} \label{eq6.2}
    ( f_i g_i ) v_j = \tau_0 ( f_i , g_j ) \tau_0 ( f_j, g_i ) v_j ( f_i g_i )
  \end{equation}
  in $ \cmdcalD $ implies $ v_j \in \cmdfraka $.
  Similarly, $ w_j \in \cmdfraka $, whence $ f_j g_j - 1 \in \cmdfraka $.
  This argument implies (b) above.

  Next, suppose that $ \cmdfraka $ does not contain any $ f_i g_i - 1 $.
  Suppose that $ v = \sum_{ j } c_j v_j ( \in V ) $ is in $ \cmdfraka $, where $ c_j \in k $.
  Then for each $i$, $ \cmdfraka $ contains $ g^{-1}_i w_i v - v g^{-1}_i w_i $,
  which equals $ c_i \tau ( f_i , g_i ) g^{-1}_i ( f_i g_i -1 ) $.
  This implies $ c_i = 0 $ for all $i$, whence $ V \cap \cmdfraka = 0 $.
  Similarly, $ W \cap \cmdfraka = 0 $.
  By (ii), (\ref{eq5.6}) is now satisfied,
  %6-4
  and $ V' = 0 = W' $.
  Theorem \ref{5.2} then implies (a).

  Case (i).
  Suppose $ f_1 g_1 - 1 \in \cmdfraka $.
  Then the relation (\ref{eq6.2}) with $ i = j = 1 $,
  together with an analogous one, proves $ v_1 \in \cmdfraka $, $ w_1 \in \cmdfraka $,
  whence we have (b).
  Suppose $ f_1 g_1 - 1 \notin \cmdfraka $.
  Then (\ref{eq3.8}) with $ i = j = 1 $ implies $ v_1 \notin \cmdfraka , w_1 \notin \cmdfraka $.
  Since (\ref{eq5.6}) is satisfied and $ V' = 0 = W' $,
  Theorem \ref{5.2} again implies (a).
\end{proof}

\begin{example} \label{6.2}
  Suppose $ k = \cmdbbQ ( q ) $ with $ q $ transcendental over $ \cmdbbQ $.
  Let $ U_q $ denote the quantized enveloping algebra, as defined by Kang \cite{K},
%6-5
  associated to a symmetrizable Borcherds-Cartan matrix $ \cmdbbA $ and a sequence $ \cmdbbm $ of positive integers.
  We will use the same notation as used in \cite[Sect.4]{M2}.
  In particular, $ U^0 $ denotes the group algebra spanned by the grouplikes $ q^h ( h \in P^{\vee} ) $,
  and we have the two objects
  \begin{equation*}
    V = \bigoplus_{ j , l } k f'_{ j l } , \quad W = \bigoplus_{ j , l } k e_{ j l } \quad \mbox{ in } {}^{ U^0 }_{ U^0 } \cmdcalYD,
  \end{equation*}
  where $ ( j , l ) $ runs through the index set
  \begin{equation*}
    \cmdcalI = \{ ( j , l ) \mid j \in I , \ 1 \leq l \leq m_j \}.
  \end{equation*}
  The Nichols algebras of $V$, $W$ are made into $ U^{ \leq 0 } $, $ U^{ \geq 0 } $ by bosonization.
  We have a non-degenerate skew pairing $ \tau : U^{ \leq 0 } \otimes U^{ \geq 0 } \rightarrow k $,
  %6-6
  due to Kang and Tanisaki \cite{KT}, as given on Page 2214, lines -4 to -1 of \cite{M2}.
  These fit in with our situation of Section 3,
  supposing $ k F = k G = U^0 $, $ H = U^{ \leq 0 } $ , $ B = U^{ \geq 0 } $;
  note that the $R$ and the $S$ are then both Nichols.
  The associated, generalized quantum double $ U^{ \leq 0 } \cmdbicross_{\tau} U^{ \geq 0 } $,
  divided by the Hopf ideal generated by $ q^h \otimes 1 - 1 \otimes q^h ( h\in P^{\vee} ) $, turns into $ U_q $.
  Since $ U^0 $ coacts on the basis elements $ f'_{ j l } $, $ e_{ i j } $ so that $ f'_{ j l } \mapsto K_j \otimes f'_{ j l } , \ e_{ j l } \mapsto K_j \otimes e_{ j l } $,
  the condition (\ref{eq6.1}) now reads
  \begin{equation*}
    q^{ - 2 s_i a_{ i j } } = \tau ( K_i , K_j ) \tau ( K_j , K_i ) \ne 1,
  \end{equation*}
  %6-7
  which is equivalent to $ a_{ i j } \ne 0 $.
  Note that the grouplike $ q^h \otimes 1 $ (or $ 1 \otimes q^h $) in $ U^{ \leq 0 } \cmdbicross_{\tau} U^{ \geq 0 } $ is central if and only if
  \begin{equation*}
    q^{ - \alpha_j ( h ) } = \tau ( q^h , K_j ) = 1 \quad  ( or~\alpha_j ( h ) = 0 )
  \end{equation*}
   for all $ j \in I $.
  In view of this, we set
  \begin{equation} \label{eq6.3}
    N = \{ h \in P^{\vee} \mid \alpha_j ( h ) = 0 \ ( j \in I ) \}.
  \end{equation}
  Recall that the number $ \left| I \right|$ of $I$, finite or countably infinite, equals the size of the matrix $ \cmdbbA $.
  Assume
  \begin{itemize}
    \item [(i)] $ \left| I \right| = 1 $ and $ \cmdbbA \ne 0 $, or
    \item [(ii)] $ \left| I \right| > 1 $ and $ \cmdbbA $ is irreducible (or the Dynkin diagram of $ \cmdbbA $ is connected).
  \end{itemize}
  Then we see from Theorem \ref{6.1} that every Hopf ideal of $ U_q $
  \begin{itemize}
    \item [(a)] is generated by all $ q^h - 1 $,
    where $h$ runs through a certain additive subgroup of N (see (\ref{eq6.3})), or
    \item [(b)] contains all $ e_{ j l }, f_{ j l }, K^2_j - 1 $,
    where $ j \in I $, $ 1 \leq l \leq m_j $.
  \end{itemize}
\end{example}

% 7-1
\section{Minimal quasitriangular pointed Hopf algebras in characteristic~ zero}

We assume that $k$ is an algebraically closed field of characteristic $ \cmdch k $ is zero, unless otherwise stated.

Suppose that we are in the situation of Remark \ref{3.4} so that $ \cmdcalD = D ( H ) $,
Drinfeld's quantum double of $ H = \cmdfrakB ( W ) \cmddotrtimes G $.

\begin{remark} \label{7.1}
  In particular, we suppose $ S = \cmdfrakB ( W ) $,
  and that it is finite-dimensional.
  By the assumption $ \cmdch k = 0 $, we see from Remark \ref{2.1} the following.
  \begin{equation} \label{eq7.1}
    \tau_0 ( f_i , g_i ) \ne 1 \quad ( i \in \cmdcalI ).
  \end{equation}
  \begin{equation} \label{eq7.2}
    \mbox{If } \left| G \right| \mbox{ is odd, } \ \tau_0 ( f_i , g_i ) \ne -1 \ ( i \in \cmdcalI ), \mbox{ whence } W' = 0.
  \end{equation}
\end{remark}

Note that (\ref{eq7.1}) implies (\ref{eq5.6}).
Recall that $ \cmdcalD = D ( H ) $ has a canonical $R$-matrix with which $ \cmdcalD $ is a minimal quasitriangular (in the sense of Radford \cite{R2}) pointed Hopf algebras.
% 7-2
Choose arbitrarily a thin Hopf ideal $ \cmdfraka $ of $ \cmdcalD $,
which is necessarily of the form $ \cmdfraka ( T , Z , \zeta ) $ (see (\ref{5.2})) by Theorem \ref{5.4}.
Set $ \cmdcalA = \cmdcalD / \cmdfraka $.
Note that
\begin{equation*}
  F \cap C = \{ 1 \} = G \cap C,
\end{equation*}
since $ \tau_0 : F \times G \rightarrow k^{\times} $ is supposed to be non-degenerate.
It follows by the proof of Theorem \ref{5.4} that $ \cmdcalA $ is of the form $ \cmdcalA ( T, Z , \zeta ) $ (see (\ref{eq5.4})),
and it includes $ B = \cmdfrakB ( V ) \cmddotrtimes F $, $ H = \cmdfrakB ( W ) \cmddotrtimes G $.
By \cite[Theorem 1]{R2}, we conclude that $ \cmdcalA $ is a minimal quasitriangular pointed Hopf algebra with respect to the $R$-matrix inherited from $ \cmdcalD = D ( H ) $.
% 7-3
Note that $ \cmdcalA $ is generated by skew primitives.

\begin{theorem} \label{7.2}
  Conversely, every minimal quasitriangular pointed Hopf algebra that is generated by skew primitives is of the form $D(H)/\cmdfraka ( T , Z , \zeta )$, where $ H = \cmdfrakB ( W ) \cmddotrtimes G $.
\end{theorem}
% 7-4
\begin{proof}
  We will use Lemmas \ref{7.4} and \ref{7.5} which will be formulated and proved later.
  
  Let $A$ be a minimal quasitriangular Hopf algebra.
  By \cite[Theorem 2]{R2}, we have a Hopf algebra quotient
  \begin{equation} \label{eq7.3}
    \pi : D ( H ) = B \cmdbicross_{\tau} H \rightarrow A
  \end{equation}
  of some quantum double $ D ( H ) $ such that $ \pi \! \mid_B $, $ \pi \! \mid_H $ are both embeddings.
  Here we suppose that $ D ( H ) $ is constructed by some finite-dimensional Hopf algebras $B$, $H$ and a non-degenerate 
  skew pairing $ \tau : B \otimes H \rightarrow k $; see \cite[Sect. 2]{DT}.
  Let
% 7-5
  \begin{eqnarray}
    \tau^l : B^{\cmdcop} \cmdarrow^{\simeq} H^* , \quad \tau^l ( b ) ( h ) = \tau ( b , h ), \label{eq7.4} \\
    \tau^r : H \cmdarrow^{\simeq} ( B^{\cmdcop} )^* , \quad \tau^r ( h ) ( b ) = \tau ( b , h ), \label{eq7.5}
  \end{eqnarray}
  denote the Hopf algebra isomorphisms associated to $ \tau $.
  
  Assume that $A$ is pointed.
  Then, $B$ and $H$ are pointed, too.
  Set $ F = G ( B ) $, $ G = G ( H ) $.
  Then the kernel of the restriction map $ H^* \rightarrow ( kG )^* $ is the Jacobson radical of $ H^* $,
  in which $ \tau^l ( f - 1 ) $ cannot be contained if $ 1 \ne f \in F $,
  since $ f - 1 $ cannot be nilpotent.
  Therefore, $ \tau^l $ induces a Hopf algebra mono $ kF \hookrightarrow ( kG )^* $.
  Similarly, $ \tau^r $ induces a mono $ kG \hookrightarrow ( kF )^* $.
  The two monos, being dual to each other, must be isomorphisms.
  Hence, $F$ and $G$ are abelian,
% 7-6
  and $ \tau $ restricts to a non-degerate pairing $ F \times G \rightarrow k^{\times} $.
  Since the maps
  \begin{equation*}
    B \cmdarrow^{\simeq}_{\tau^l} H^* \cmdarrow_{\cmdres} ( kG )^*, \quad H \cmdarrow^{\simeq}_{\tau^r} B^* \cmdarrow_{\cmdres} ( kF )^*,
  \end{equation*}
  combined with the isomorphisms just obtained, give Hopf algebra retractions $ B \rightarrow kF $, $ H \rightarrow kG $,
  we see from \cite[Theorem 3]{R1} that
  \begin{equation*}
    B = R \cmddotrtimes F, \quad H = S \cmddotrtimes G,
  \end{equation*}
  where $R$, $S$ are local irreducible braided Hopf algebras in $ {}^F_F \cmdcalYD $, $ {}^G_G \cmdcalYD $, respectively.
  Set
  \begin{equation*}
    V = P ( R ) \ ( = \cmdYD ( B ) ) , \quad W = P ( S ) \ ( = \cmdYD ( H ) ).
  \end{equation*}
  These are of diagonal type by our assumption on $k$.
  Since the augmentation ideals $ R^+ $, $ S^+ $ are nilpotent, we have
% 7-7
  \begin{equation*}
    \tau ( f , s ) = \varepsilon ( s ) , \quad \tau ( r , g ) = \varepsilon ( r ) \quad ( f \in F , g \in G , r \in R , s \in S ).
  \end{equation*}
  It follows by Lemma \ref{7.5} (1) that in $ D(H) $, $V$ and $W$ are stable under conjugation by $G$, $F$, respectively, so that
  \begin{eqnarray}
    g v g^{-1} &=& \tau ( v_{-1} , g ) v_0 \quad ( g \in G , v \in V ), \label{eq7.6} \\
    f w f^{-1} &=& \tau ( f^{-1} , w_{-1} ) w_0 \quad ( f \in F , w \in W ). \label{eq7.7}
  \end{eqnarray}
  Set $ \Gamma = F G $ in $A$.
  Then, $ \Gamma = G ( A ) $.
  By (\ref{eq7.6}), (\ref{eq7.7}), $V$ and $W$, embedded into $ \cmdYD ( A ) $ (see Lemma \ref{2.2}),
  are stable under $ \Gamma $-action.
  
  Assume that $A$ is generated by skew primitives.
  By the $ \Gamma $-stability shown above,
%7-8
  Lemma \ref{7.4} implies that $R$, $S$ are generated by $V$, $W$, respectively.
  By Lemma \ref{7.5} (2), (3), $R$ and $S$ are Nichols,
  and $ \tau \! \mid_{V \otimes W} : V \otimes W \rightarrow k $ is non-degenerate.
  It follows that $ D ( H ) $ is so as given at the begininng of this section.
  Theorem \ref{7.2} now follows by Theorem \ref{5.4}.
\end{proof}

\begin{corollary} \label{7.3}
  If $A$ is a minimal quasitriangular pointed Hopf algebra such that the prime divisors of $ \left| G(A) \right| $
  are all greater than $7$, then it is of the form $ D(H)/\cmdfraka ( T , Z , \zeta )$ with $ H = \cmdfrakB ( W ) \cmddotrtimes G $,
  in which $ \cmdfraka ( T , Z , \zeta ) $ is skinny, that is, $Z$ and $ \zeta $ are zero.
\end{corollary} 

\begin{proof}
  As was proved by Andruskiewitsch and Schneider \cite[Theorem 5.5]{AS2}, $A$ is necessarily generated by 
  skew-primitives under the assumption. Hence the conclusion follows by Theorem \ref{7.2} and \eqref{7.2}.
\end{proof}

The referee kindly pointed out that our proof of Theorem \ref{7.2} answers in the positive a folklore 
question, as is formulated by the following corollary, under the hypothesis of generation by skew-primitives.

\begin{corollary} \label{7.3a}
  Let $A$ be a finite-dimensional pointed Hopf algebra whose dual $A^{*}$ as well is pointed. 
  Then $A$ and $A^{*}$ are both bosonizations of Nichols algebras by finite abelian groups, provided
  that they are generated by skew-primitives.
\end{corollary} 

\begin{proof}
  This indeed follows by the proof of Theorem \ref{7.2} since we may replace the $H$ and the $B$ of the proof
  with the present $A$ and $(A^{*})^{\mathrm{cop}}$.
  \end{proof}

For the following Lemmas \ref{7.4} and \ref{7.5}, which we have used above, $k$ may be an arbitrary field.

\begin{lemma} \label{7.4}
  Let $A$ be a pointed Hopf algebra with $ U = \cmdYD ( A ) $ ($ \in {}^{G ( A )}_{G ( A )} \cmdcalYD $).
  Assume that $U$ is semisimple in $ {}^{G ( A )}_{G ( A )} \cmdcalYD $;
  this holds true if $ \cmdch k = 0 $ and $ G ( A ) $ is finite.
% 7-9
  Let $ B \subset A $ be a Hopf subalgebra with $ V = \cmdYD ( B ) $ ($ \in {}^{G ( B )}_{G ( B )} \cmdcalYD $);
  then $ V \subset U $ (see Lemma \ref{2.2}).
  Assume that $V$ is stable in $U$ under the $ G ( A ) $-action.
  If $A$ is generated by skew primitives, then $B$ is, too.
\end{lemma}

\begin{proof}
  We may replace $A$, $B$ with the associated graded Hopf algebras, so that
  \begin{equation*}
    A = \cmdfrakB ( U ) \cmddotrtimes G ( A ) , \quad B = S \cmddotrtimes G ( B ), 
  \end{equation*}
  where $S$ is a braided graded Hopf algebra in $ {}^{G ( B )}_{G ( B )} \cmdcalYD $ such that $ S ( 0 ) = k $, $ S ( 1 ) = V $.
  Note that $V$ generates the Nichols algebra $ \cmdfrakB ( V ) $ in $S$.
  We need to prove $ \cmdfrakB ( V ) = S $.

  By the assumptions given above, $V$ is a subobject of $U$ in $ {}^{G ( A )}_{G ( A )} \cmdcalYD $,
  and we have a retraction $ U \rightarrow V $,
% 7-10
  which uniquely extends to a retraction $ \tau : \cmdfrakB ( U ) \rightarrow \cmdfrakB ( V ) $ of braided graded Hopf algebras.
  Note that Radford's Theorem \cite[Theorem 3]{R1} can be proved, by rephrasing his proof word by word, in the genralized, braided context.
  Apply the thus generalized theorem to $ \cmdfrakB ( V ) \subset S $ with the retraction $ \pi \! \mid_S : S \rightarrow \cmdfrakB ( V ) $.
  Then we obtain an isomorphism
  \begin{equation} \label{eq7.8}
    S \simeq Q \cmddotrtimes \cmdfrakB ( V ).
  \end{equation}
  Regarded just as a coalgebra,
  the right-hand side of (\ref{eq7.8}) denotes the smash coproduct in the braided context,
  associated to a braided left $ \cmdfrakB ( V ) $-comodule coalgebra $Q$ ($ = S / S \cmdfrakB ( V )^+ $) in $ {}^{G ( B )}_{G ( B )} \cmdcalYD $.
  If $ \cmdfrakB ( V ) \subsetneqq S $ or $ k \subsetneqq Q $, then $Q$ contains a non-zero primitive,
  say $u$, on which the irreducible $ \cmdfrakB ( V ) $ coacts trivially.
% 7-11
  This $u$ is a primitive in $S$ which is not contained in $V$; this is absurd.
  Therefore, $ \cmdfrakB ( V ) = S $.
\end{proof}

\begin{lemma} \label{7.5}
  Let $J$, $K$, be Hopf algebras with bijective antipode.
  Let $ R~( \in {}^J_J \cmdcalYD ) $, $ S~( \in {}^K_K \cmdcalYD ) $ be braided Hopf algebras.
  Let
  \begin{equation*}
    \tau : ( R \cmddotrtimes J ) \otimes ( S \cmddotrtimes K ) \rightarrow k
  \end{equation*}
  be a skew pairing, and let
  \begin{equation*}
    \cmdcalD = ( R \cmddotrtimes J ) \cmdbicross_{\tau} ( S \cmddotrtimes K )
  \end{equation*}
  denote the associated generalized quantum double,
  which is constructed just as the $ \cmdcalD $ in (\ref{eq3.4}); see \cite[Sect. 2]{DT}.
  Set
  \begin{equation*}
    V = P ( R ) \ ( \in {}^J_J \cmdcalYD ) , \quad W = P ( S ) \ ( \in {}^K_K \cmdcalYD ).
  \end{equation*}
  \begin{enumerate}
    \item Assume that
    \begin{eqnarray}
      \tau ( a , s ) &=& \varepsilon ( a ) \varepsilon ( s ) \quad ( a \in J , s \in S ), \label{eq7.9} \\
      \tau ( r , x ) &=& \varepsilon ( r ) \varepsilon ( x ) \quad ( x \in K , r \in R ). \label{eq7.10}
    \end{eqnarray}
    Then in $ \cmdcalD $,
    $V$ and $W$ are stable under conjugation by $K$, $J$, respectively, so that
    \begin{eqnarray*}
      x_1 v \cmdcalS ( x_2 ) &=& \tau ( v_{-1} , x ) v_0 \quad ( x \in K , v \in V ),\\
      a_1 w \cmdcalS ( a_2 ) &=& \tau ( \cmdcalS ( a ) , w_{-1} ) w_0 \quad ( a \in J , w \in W ).
    \end{eqnarray*}
    See (\ref{eq2.1}), (\ref{eq2.2}) for the notation used above.
    \item Assume (\ref{eq7.9}), (\ref{eq7.10}).
    Let
    \begin{equation*}
      \Delta ( t ) = t_{(1)} \otimes t_{(2)} \quad ( t \in R \mbox{ or } t \in S )
    \end{equation*}
    denote the coproducts of the braided coalgebras $R$, $S$.
    Let $ \bar{\cmdcalS} $ denote the composite-inverse of the antipode $ \cmdcalS $ of $ R \cmddotrtimes J $.
    Then,
    \begin{eqnarray}
      \tau ( r , s s^{\prime} ) &=& \tau ( r_{(1)} , s^{\prime} ) \tau ( r_{(2)} , s ), \label{eq7.11} \\
      \tau ( \bar{\cmdcalS} ( r ) \bar{\cmdcalS} ( r^{\prime} ) , s ) &=& \tau ( \bar{\cmdcalS} ( r ) , s_{(1)} ) \tau ( \bar{\cmdcalS} ( r^{\prime} ) , s_{(2)} ), \label{eq7.12} \\
% 7-13
      \tau ( \bar{\cmdcalS} ( r ) a , s x ) &=& \tau ( \bar{\cmdcalS} ( r ) , s ) \tau ( a , x ), \label{eq7.13}
    \end{eqnarray}
    where $ a \in J $, $ x \in K $, $ r \in R $, $ s \in S$.
    Moreover, $ \tau $ is non-degenerate if and only if the restrictions
    \begin{equation*}
      \tau \! \mid_{J \otimes K} : J \otimes K \rightarrow k , \quad \tau \! \mid_{R \otimes S} : R \otimes S \rightarrow k
    \end{equation*}
    are both non-degenerate.
    \item Assume that $R$, $S$ are generated by $V$, $W$, respectively, and that
    \begin{equation} \label{eq7.14}
      \tau ( a , w ) = 0 = \tau ( v , x ) \quad ( a \in J , x \in K , v \in V , w \in W ).
    \end{equation}
    If $ \tau \! \mid_{J \otimes K} $ is non-degenerate, then the following are equivalent to each other:
    \begin{itemize}
      \item [(a)] $ \tau \! \mid_{R \otimes S} $ is non-degenerate;
      \item [(b)] $ \tau \! \mid_{V \otimes W} $ is non-degenerate, and $R$, $S$ are both Nichols
% 7-14
      (or more precisely, the inclusions $ V \hookrightarrow R $, $ W \hookrightarrow S $ uniquely extend to isomorphisms, $ \displaystyle \cmdfrakB ( V ) \cmdarrow^{\simeq} R $, $ \displaystyle \cmdfrakB ( W ) \cmdarrow^{\simeq} S $);
      \item [(c)] $  \tau \! \mid_{\bar{R} \otimes S} $ is non-degenerate, where $ \bar{R} = \bar{\cmdcalS} ( R ) $;
      \item [(e)] $ \tau $ is non-degenerate.
    \end{itemize}
  \end{enumerate}
\end{lemma}

\begin{proof}
  (1) This is directly verified.
  
  (2) Set $ B = R \cmddotrtimes J $, $ H = S \cmddotrtimes K $.
  We can prove, in the present situation, the same formulae as those given in \cite[Lemma 2.2]{M2},
  which includes (\ref{eq7.11}), just in the same way.
  The result can apply to the skew pairing
  \begin{equation} \label{eq7.15}
    H \otimes B \rightarrow k , \quad h \otimes b \mapsto \tau^{-1} ( b , h ),
  \end{equation}
  where $ \tau^{-1} $ denotes the convolution-inverse of $ \tau $.
% 7-15
  Obtained are the same formulae as those given in \cite[Lemma 2.3]{M2}, which include (\ref{eq7.12}), (\ref{eq7.13}).
  
  We see from (\ref{eq7.13}) that $ \tau $ is non-degenerate if and only if $ \tau \! \mid_{J \otimes K} $ and $ \tau \! \mid_{ \bar{R} \otimes S } $ are both non-degenerate, where $ \bar{R} = \bar{\cmdcalS} ( R ) $.
  Apply this result to the skew pairing (\ref{eq7.15}),
  using $ \tau^{-1} ( b , h ) = \tau ( b , \cmdcalS ( h ) ) $.
  Then we see that $ \tau^{-1} $ is non-degenerate if and only if $ \tau \! \mid_{J \otimes K} $, $ \tau \! \mid_{R \otimes S} $ are so.
  The desired equivalence follows since $ \tau $ and $ \tau^{-1} $ are non-degenerate at the same time.
  
  (3) We remark that (a)-(c) above are the same as those given in \cite[Proposition 2.6(1)]{M2},
% 7-16
  in which these conditions together with another (d) are proved to be equivalent,
  in a more restricted situation,
  but without assuming that $ \tau \! \mid_{J \otimes K} $ is non-degenerate.
  By the proof of (2) above,
  we see (a) $ \Leftrightarrow $ (c) $ \Leftrightarrow $ (e).
  Note that if $ v \in V $, then $ \bar{\cmdcalS} ( v ) = - v_0 \bar{\cmdcalS} ( v_{-1} ) $,
  and so $ \tau ( \bar{\cmdcalS} ( v ) , s ) = - \tau ( v , s ) $ for $ s \in S $; see \cite[(2.12)]{M2}.
  Then we see from (\ref{eq7.11}), (\ref{eq7.12}) that (a) (or (c)) implies that $ \tau \! \mid_{V \otimes W} $ is non-degenerate.
  
  Since $ \tau \! \mid_{V \otimes W} $ satisfies \cite[(2.1), (2.2)]{M2},
  it follows by \cite[Proposition 2.1]{M2} that $ \tau \! \mid_{J \otimes K} $, $ \tau \! \mid_{V \otimes W} $ uniquely extend to such a skew pairing
% 7-17
  \begin{equation*}
    \tilde{\tau} : ( T ( V ) \cmddotrtimes J ) \otimes ( T ( W ) \cmddotrtimes K ) \rightarrow k
  \end{equation*}
  that satisfies the same assumption as (\ref{eq7.14}).
  Obviously, $ \tilde{\tau} $ factors through our $ \tau $.
  Assume that $ \tau \! \mid_{V \otimes W} $ ($ = \tilde{ \tau } \! \mid_{V \otimes W} $) is non-degenerate.
  It then follows by \cite[Proposition 2.6(1), (b) $ \Rightarrow $ (a)]{M2} that the radicals of $ \tilde{\tau} \! \mid_{T ( V ) \otimes T ( W )} $ coincide the ideals defining $ \cmdfrakB ( V ) $, $ \cmdfrakB ( W ) $.
  Therefore, (a) $ \Rightarrow $ (b).
  The converse is proved in \cite{M2}.
\end{proof}

%8-1
\section{Gelaki's classification of minimal triangular pointed Hopf~ algebras}

We will reproduce from our Theorem \ref{7.2} Gelaki's classification results \cite[Theorem 4.4, Theorem 5.1]{G} of minimal triangular pointed Hopf algebras in characteristic zero.

\subsection{}

In this subsection, $k$ may be arbitrary.
Let $A$ be a minimal quasitriangular Hopf algebra with $R$-matrix $ \cmdcalR $.
Then we have as in (\ref{eq7.3}), a quotient
\begin{equation*}
  \pi : D ( H ) = B \cmdbicross_{\tau} H \rightarrow ( A , \cmdcalR )
\end{equation*}
of some quantum double $ D ( H ) $ such that $ \pi \! \mid_B $, $ \pi \! \mid_H $ are embeddings.
Assume that $ \pi \! \mid_B : B \rightarrow A $, $ \pi \! \mid_H : H \rightarrow A $ are isomorphisms.
This assumption is necessarily satisfied if $ ( A , \cmdcalR ) $ is minimal triangular.
Let us identify $B$, $H$ both with $A$ via the isomorphisms.
Then the non-degenerate skew pairing $ \tau : A \otimes A \rightarrow A $ is a co-quasitriangular structure on $A$,
whence the element $ \tau^* $ in $ A^* \otimes A^* $ is an $R$-matrix of $ A^* $.
Moreover the Hopf algebra map \cite[p.3]{G}
\begin{equation*}
  f_{\cmdcalR} : ( A^* , \tau^* ) \rightarrow ( A^\cmdcop , \cmdcalR_{2 1} ) , \ f_\cmdcalR ( p ) = ( p \otimes \cmdid ) ( \cmdcalR )
\end{equation*}
arising from $\cmdcalR$ is an isomorphism with inverse $ \tau^l $; see (\ref{eq7.4}).
From this the next lemma follows.

\begin{lemma} \label{8.1}
  $ ( A , \cmdcalR ) $ is triangular if and only if
  \begin{equation*}
    \tau ( a_1 , b_1 ) \tau ( b_2 , a_2 ) = \varepsilon ( a ) \varepsilon ( b ) \quad ( a , b \in A )
  \end{equation*}
\end{lemma}

\subsection{}

In what follows we assume that $k$ is an algebraically closed field of characteristic zero.
% 8-3
Let $ ( A , \cmdcalR ) $ be a minimal triangular pointed Hopf algebra.
By \cite[Theorem 6.1]{AEG}, $A$ is generated by skew primitives.
Since it is in particular quasitriangular,
it follows by Theorem \ref{7.2} that $ ( A , \cmdcalR ) $ is isomorphic to some $D(H)/\cmdfraka ( T , Z , \zeta )$, 
where $ H = \cmdfrakB ( W ) \cmddotrtimes G $. 
Here we suppose as before that $D(H)$ is presented as a generalized quantum double
\begin{equation} \label{eq8.1}
  D(H) = ( \cmdfrakB ( V ) \cmddotrtimes F ) \cmdbicross_{\tau} ( \cmdfrakB ( W ) \cmddotrtimes G ),
\end{equation}
for which the index set $ \cmdcalI $ as well as the abelian groups $F$, $G$ are supposed to be finite,
and the skew-pairing $\tau$ to be non-degenerate.

% 8-4
Let $D(H)$, $\cmdfraka ( T , Z , \zeta )$ be as above.
We wish to know the conditions under which $D(H)/ \cmdfraka ( T , Z , \zeta )$ is indeed finite-dimensional and triangular.
As is seen from the argument in Section 8.1,
we may suppose by replacing $ \tau_0 ( = \tau \! \mid_{F \times G} ) $ that $F=G$,
the non-degenerate bimultiplicative map $ \tau_0 : G \times G \rightarrow k^{\times} $ is skew-symmetric in the sense
\begin{equation} \label{eq8.3}
  \tau_0 ( f , g ) = \tau_0 ( g , f )^{-1} \quad ( f , g \in G ),
\end{equation}
and the $T$ in $\cmdfraka ( T , Z , \zeta )$ is given by
\begin{equation*}
  T = \{ ( g^{-1} , g ) \in G \times G \mid g \in G \}.
\end{equation*}
In order that $ T \subset C $, we require
\begin{equation*}
  f_i = g^{-1}_i \quad ( i \in \cmdcalI ).
\end{equation*}
% 8-5
By (\ref{eq8.3}), we have $ \cmdcalI_T = \cmdcalI $, whence $ V_T = V $,
and we can take $V$ as the $Z$ in $\cmdfraka ( T , Z , \zeta )$, as we should.
Note that $ \Gamma = G \times G / T $ is naturally identified with $G$.
Then one sees that as objects in $ {}^\Gamma_\Gamma \cmdcalYD $,
\begin{equation} \label{eq8.4}
  V = \bigoplus_{i \in \cmdcalI} ( k v_i ; \widehat{g}^{-1}_i , g^{-1}_i ) , \quad W = \bigoplus_{i \in \cmdcalI} ( kw_i ; \widehat{g}_i , g_i ),
\end{equation}
where $ \widehat{g}_i = \tau_0 ( g_i, \ \ ) $; see Proposition \ref{3.5}.
Note that $ \tau_0 ( g_i , g_i ) = \pm 1 $ by (\ref{eq8.3}).
In order that $ \cmdfrakB ( V ) $, $ \cmdfrakB ( W ) $ are finite-dimensional, we require
\begin{equation*}
  \tau_0 ( g_i , g_i ) = -1 \quad ( i \in \cmdcalI );
\end{equation*}
see Remark \ref{7.1}.
Let
\begin{equation*}
  I_{\tau_0} = \{ g \in G \mid \tau_0 ( g , g ) = -1 \},
\end{equation*}
% 8-6
as in \cite{G} (in which our $\tau_0$, $I_{\tau_0}$ are denoted by $F$, $I_F$).
For each $ g \in I_{\tau_0} $, let
\begin{eqnarray*}
  \cmdcalI_g = \{ i \in \cmdcalI \mid g_i = g \}, \quad n_g = \left| \cmdcalI_g \right| \ ( \geq 0 ), \\
  V_g = \bigoplus_{i \in \cmdcalI_g} ( kv_i ; \widehat{g}^{-1}_i , g^{-1}_i ) , \quad W_g = \bigoplus_{i \in \cmdcalI_g} ( kw_i ; \widehat{g}_i , g_i ).
\end{eqnarray*}
Note that $ \cmdfrakB ( W_g ) $ equals the exterior algebra $ \bigwedge ( W_g ) $ of $ W_g $, whence
\begin{equation} \label{eq8.5}
  H = \cmdfrakB ( W ) \cmddotrtimes G = ( \cmdbarbigotimes_{g \in I_{\tau_0}} \bigwedge ( W_g ) ) \cmddotrtimes G.
\end{equation}

One sees from (\ref{eq8.4}) that there exists an isomorphism $ \displaystyle \zeta : V \cmdarrow^{\simeq} W $ in $ {}^\Gamma_\Gamma \cmdYD $ if and only if
\begin{equation} \label{eq8.6}
  n_g = n_{g^{-1}} \quad ( g \in I_{\tau_0} ) ,
\end{equation}
% 8-7
which we now assume to hold.
Such an isomorphism $\zeta$ is precisely a direct sum of linear isomorphisms $ \displaystyle M_g : V_g \cmdarrow^{\simeq} W_{g^{-1}} $,
where $ g \in I_{\tau_0} $.
We understand that $ M_g $ is represented as an $ n_g \times n_g $ invertible matrix with respect to the bases $ ( v_i \mid i \in \cmdcalI_g ) $, $ ( w_i \mid i \in \cmdcalI_{g^{-1}} ) $.

\begin{proposition} \label{8.2}
  Suppose $Z = V$, and that $\zeta$ is given by $ M_g~( g \in I_{\tau_0} ) $. Set $ \cmdfraka = \cmdfraka ( T , V , \zeta )$.
  We already know that $ D(H)/ \cmdfraka $, given the $R$-matrix inherited from $D(H)$, is a minimal quasitriangular Hopf algebra.
  \begin{enumerate}
    \item $ D(H)/ \cmdfraka $ is triangular if and only if
    \begin{equation} \label{eq8.7}
      M_{g^{-1}} = {}^t M_g \quad ( g \in \cmdcalI_{\tau_0} ).
    \end{equation}
% 8-8
    \item (Gelaki) Every minimal triangular pointed Hopf algebra is of this form $ D(H)/ \cmdfraka $ with (\ref{eq8.7}) satisfied.
  \end{enumerate}
\end{proposition}

\begin{proof}
  (1) One sees from Lemma \ref{8.1} that $ D(H)/ \cmdfraka $ is triangular if and only if
  \begin{equation*}
    \tau ( v_i , \zeta ( v_j ) ) + \tau ( g^{-1}_j , g_i ) \tau ( v_j , \zeta ( v_i ) ) = 0 \quad ( i , j \in \cmdcalI ).
  \end{equation*}
  This last condition is equivalent to (\ref{eq8.7}).

  (2) This follows from (1) and the argument above.
\end{proof}

\begin{remark} \label{8.3}
  Keep the notation as above.
  \begin{enumerate}
    \item $ D(H)/ \cmdfraka $ is, as a Hopf algebra, isomorphic to the one given in (\ref{eq8.5}).
    The latter is precisely Gelaki's $ H ( \cmdcalD ) $ which is associated to the data,
    denoted by $ \cmdcalD $ in \cite{G},
    consisting of $G$, $\tau_0$, $(n_g)$ (our $ W_g $ is denoted by $V_g$ in \cite{G}).
% 8-8a
    Part 2 of the preceding proposition essentially coincides with Gelaki's classification \cite[Theorem 5.1]{G} of minimal triangular pointed Hopf algebras.
% 8-9
    \item Let $ B = \cmdfrakB ( V ) \cmddotrtimes G $, $ H = \cmdfrakB ( W ) \cmddotrtimes G $.
    By (1), $ H = D(H)/ \cmdfraka = H ( \cmdcalD ) $ as Hopf algebras.
    We see by modifying the argument above that the minimal triangular structures on $H$ are in 1-1 correspondence with the pairs $ ( \varphi , ( M_g ) ) $ consisting of such a group isomorphism $ \displaystyle \varphi : G \cmdarrow^{\simeq} G $ and linear isomorphisms $ \displaystyle M_g : V_g \cmdarrow^{\simeq} W_{g^{-1}} $ that satisfy (\ref{eq8.7}) and
    \begin{equation*}
      \varphi ( g_i ) = g_i \ ( i \in \cmdcalI ) , \quad \tau_0 ( f , \varphi ( g ) ) = \tau_0 ( \varphi ( f ) , g ) \ ( f , g \in G ).
    \end{equation*}
    To such a pair is associated the structure which is inherited from $ D ( H ) = B \cmdbicross_\tau H $ via the Hopf algebra map $ \pi : B \cmdbicross_\tau H \rightarrow H $ determined by
% 8-10
    \begin{equation*}
      \pi \! \mid_H = \cmdid {}_H , \quad \pi ( g \otimes 1 ) = \varphi ( g ) \ ( g \in G ) , \quad \pi ( v ) = M_g v \ ( v \in V_g ).
    \end{equation*}
    Interpret this result by regarding $ \varphi $, $ M_g $ as isomorphisms $ \displaystyle \widehat{G} \cmdarrow^{\simeq} G $, $ \displaystyle W^*_g \cmdarrow^{\simeq} W_{g^{-1}} $ via the identifications $ G = \widehat{G} $, $ V_g = W^*_g $ given by $ \tau^l $.
    The obtained result coincides with Gelaki's classification \cite[Theorem 4.4]{G} of minimal triangular structures on $ H ( \cmdcalD ) $.
    We see from our construction that if we fix $ \varphi $,
    two systems $ ( M_g ) $, $ ( M'_g ) $ of linear isomorphisms give isomorphic minimal triangular structures on $ H = H ( \cmdcalD ) $.
  \end{enumerate}
\end{remark}

\section*{Acknowledgements}
  I gratefully acknowledge that this work was supported by Grant-in-Aid for Scientific Research (C) 20540036 from the Japan Society for the Promotion of Science,
  and by Renewed-Research-Study Program of the Alexander von Humboldt Foundation.
  Part of this work was done while I was staying at the LMU of Munich, June-July 2009.
  I thank Prof. Hans-J\"{u}rgen Schneider for his very warm hospitality. 
  I also thank the referee for his or her valuable suggestions and helpful comments.

% R-1

\end{document}